\documentclass[12pt]{amsart}
\usepackage[nobysame,abbrev,alphabetic]{amsrefs}
\usepackage{amssymb}

\DeclareMathOperator{\Hom}{Hom}
\DeclareMathOperator{\Lyn}{Lyn}
\DeclareMathOperator{\Mag}{Mag}
\DeclareMathOperator{\res}{res}
\DeclareMathOperator{\Supp}{Supp}

\newcommand{\Bock}{\mathrm{Bock}}
\newcommand{\nek}{,\ldots,}
\newcommand{\inv}{^{-1}}
\newcommand{\isom}{\cong}

\newcommand{\trg}{\mathrm{trg}}
\newcommand{\phm}{\phantom{-}}

\newtheorem{thm}{Theorem}[section]
\newtheorem{cor}[thm]{Corollary}

\newtheorem{prop}[thm]{Proposition}

\newtheorem{exam}[thm]{Example}
\newtheorem{rem}[thm]{Remark}

\newtheorem*{thmA}{Theorem A}
\newtheorem*{thmB}{Theorem B}
\numberwithin{equation}{section}

\newcommand{\alp}{\alpha}
\newcommand{\eps}{\epsilon}

\newcommand{\sig}{\sigma}
\newcommand{\dbF}{\mathbb{F}}

\newcommand{\dbU}{\mathbb{U}}
\newcommand{\dbZ}{\mathbb{Z}}

\begin{document}

\title[Higher Labute--Serre duality]{Higher Labute--Serre duality and Lyndon words}

\author{Ido Efrat}
\address{Earl Katz Family Chair in Pure Mathematics\\
Department of Mathematics\\
Ben-Gurion University of the Negev\\
P.O.\ Box 653, Be'er-Sheva 8410501\\
Israel} \email{efrat@bgu.ac.il}

\author{Levav Ferber Tas}
\address{Department of Mathematics\\
Ben-Gurion University of the Negev\\
P.O.\ Box 653, Be'er-Sheva 8410501\\
Israel} \email{levavft@gmail.com}

\thanks{This work was supported by the Israel Science Foundation (grant No.\ 569/21). \\
The paper is based on the second author's M.Sc.\ thesis, carried out at Ben-Gurion University under the supervision of the first author.}

\keywords{Labute--Serre duality, Lyndon words, transgression pairing}

\subjclass[2020]{Primary 12G05, Secondary 20J06, 68R15}

\maketitle

\begin{abstract}
Let $S$ be a free profinite group on a finite ordered basis $X$, and let $S^{(n,p)}$, $n=1,2\nek$ denote its lower $p$-central filtration. 
There is a natural duality between $S^{(n,p)}/S^{(n+1,p)}$ and $H^2(S/S^{(n,p)},\dbF_p)$.
These $\dbF_p$-linear spaces admit natural bases indexed by Lyndon words of length $\leq n$ in the alphabet $X$. 
These bases are known to be unitriangularly dual. 
We prove that they are much closer to being fully dual, by showing that the pairing between two basis elements vanishes unless the corresponding Lyndon words are permutations of one another.
We further show that the value of the pairing is essentially independent of $n$.
\end{abstract}

\section{Introduction}
\subsection{The main results}
Let $p$ be a fixed prime number.
Let $S=S_X$ be the free profinite group on a nonempty finite set $X$.
The \textsl{lower $p$-central filtration} $S^{(n,p)}$, $n=1,2\nek$ of $S$ is defined inductively by $S^{(1,p)}=S$ and $S^{(n+1,p)}=(S^{(n,p)})^p[S,S^{(n,p)}]$.
Thus $S^{(n+1,p)}$ is the closed subgroup of $S$ generated by all elements $\tau^p$ and $[\sig,\tau]=\sig\inv \tau\inv \sig\tau$ with $\sig\in S$ and $\tau\in S^{(n,p)}$.

We assume from now on that $n\geq2$.
The quotient $S^{(n,p)}/S^{(n+1,p)}$ and the cohomology group $H^2(S/S^{(n,p)},\dbF_p)$
are $\dbF_p$-linear spaces.
They are related by a canonical non-degenerate bilinear map
\[
(\cdot,\cdot)_n\colon S^{(n,p)}/S^{(n+1,p)}\times H^2(S/S^{(n,p)},\dbF_p)\to\dbF_p,
\]
called the \textsl{transgression pairing}, obtained in a natural way from the substitution pairing $S^{(n,p)}\times\Hom_{\rm cont}(S^{(n,p)},\dbF_p)\to\dbF_p$ 
 -- see \S\ref{section on the transgression pairing} for more details.

In \cite{Efrat17}, the first author constructs canonical bases $(\bar\sig_{w,n})_w$ and $(\alp_{w,n})_w$, respectively, for these two linear spaces, which are \textsl{unitriangularly dual} with respect to $(\cdot,\cdot)_n$,
that is, the matrix $[(\bar\sig_{w,n},\alp_{w',n})_n]_{w,w'}$ is unipotent upper triangular.
We call it the \textsl{fundamental matrix of depth $n$} for the pairing.

The construction of the linear bases uses tools from the \textsl{combinatorics of words}.
Specifically, one fixes a total order on $X$ and considers $X$ also as an alphabet.
Then we let $w$ range over all words $w=(x_1\cdots x_s)$ of length $1\leq s\leq n$ in this alphabet, which are \textsl{Lyndon words}, that is, $w$ is lexicographically smaller than all its proper suffixes.
We order these words length-lexicographically.
Every Lyndon word has a \textsl{standard factorization} as a concatenation of shorter Lyndon words.
Iterating this process down to the level of letters in $X$, this gives rise to 
an iterated commutator $\tau_w$ of depth $s$.
We then take $\sig_{w,n}=\tau_w^{p^{n-s}}$ (See \S\ref{section on power series} below).

The construction of the cohomology elements $\alp_{w,n}$ is more subtle and is based on the \textsl{Magnus representation} of $S$ in the group of unipotent upper triangular matrices of size $(s+1)\times(s+1)$ over the ring $\dbZ/p^{n-s+1}$.
We refer to \S\ref{section on the Magnus homomorphism} for the precise definition and only mention at this point that in the extreme cases $s=1$ and $s=n$ the basis elements $\alp_{w,n}$ are certain natural Bockstein elements or, resp., belong to a Massey product related to $w$.

Although the above bases are not dual for any $n\geq3$, in the current paper we show that they are much closer to being so than previously thought, in the sense that the failure of duality is confined to permutation classes:

\begin{thmA}
For Lyndon words $w,w'$ of length $\leq n$ in $X$ which are not permutations of one another we have $(\sig_{w,n},\alp_{w',n})_n=0$.
\end{thmA}

Therefore, if we group the Lyndon words in $X$ of length $\leq n$ according to their permutation classes, and order each class lexicographically, then the revised fundamental matrix becomes a block diagonal matrix in which the diagonal blocks are unipotent upper triangular.
Moreover, we show:

\begin{thmB}
The matrix block associated with a permutation class of a Lyndon word of length $s$ is independent of $n$, as long as $s\leq n$.
\end{thmB}

In the last section (\S\ref{section on the fundamental matrix of depth 4}) we compute these blocks explicitly for $n\leq4$.

\subsection{History and related works}
When $n=2$, the pairing $(\cdot,\cdot)_2$ has been explicitly described by Labute \cite{Labute 67} in his foundational work on the structure of pro-$p$ Demushkin groups,  following Serre \cite{SerreDemuskin}, whence the name \textsl{Labute--Serre duality}.
Here either $w=(x)$, with $x\in X$, or $w=(xy)$, with $x,y\in X$ satisfying $x<y$.
The permutation classes of Lyndon words are then singletons, so by Theorem A, the bases $\bar\sig_{w,2}$ and $\alp_{w,2}$ are in this case fully dual.

Further, when $n=2$ the basis elements take a particularly neat form:
The cosets of the elements $x$ of $X$ form a linear basis of $S/S^{(2,p)}$.
Also take the dual basis $\varphi_x$, $x\in X$, of $H^1(S/S^{(2,p)},\dbF_p)$. 
Then, on the profinite side, $\sig_{(x),2}=x^p$ and $\sig_{(xy),2}=[x,y]=x\inv y\inv xy$, whereas on the cohomology side, $\alp_{(x),2}$ is the Bockstein element $\Bock(\varphi_x)$ (See Example \ref{Bocksteins and Massey products}(a)) and $\alp_{(xy),2}$ is the cup product $\varphi_x\cup\varphi_y$.
We refer to \cite{NeukirchSchmidtWingberg}*{Ch.\ III, \S9} for a modern exposition of these facts.

The Labute--Serre duality in the case $n=2$ had numerous applications in the structure theory of profinite and pro-$p$ groups and in Galois cohomology.
See the introduction of \cite{Efrat17} and the references therein.
For more recent applications see, e.g., 
\cite{HamzaMaireMinacTan25}, \cite{MinacPasiniQuadrelliTan21}, \cite{MinacPasiniQuadrelliTan22},  \cite{MinacTan17}, \cite{Quadrelli24} (for $n=2$) or \cite{Efrat26} for all $n\geq2$.

A related filtration of $S$ is the \textsl{$p$-Zassenhaus filtration} $S_{(n,p)}$.
It is the slowest filtration by closed subgroups of $S$ satisfying $S_{(1,p)}=S$, $[S_{(n,p)},S_{(n',p)}]\subseteq S_{(n+n',p)}$, and $(S_{(n,p)})^p\subseteq S_{(pn,p)}$.
When $n=3$, Vogel \cite{VogelThesis} constructed for this filtration similar dual bases with depth-$3$ commutators and certain 3-fold Massey products.
His index set, however, differs from that of the Lyndon words.
For arbitrary $n\geq2$ and the $p$-Zassenhaus filtration, results analogous to those of \cite{Efrat17} were obtained in \cite{Efrat23}.
The subsequent work \cite{Efrat24} gives an axiomatic approach containing these two filtrations as special cases.
Analogs of Theorems A and B can also be obtained in these contexts.

\subsection{Proof idea}
To the two previous roles of $X$, as a basis of $S$ and an alphabet, we add a third role, as a set of non-commuting variables.
We then associate to the Lyndon word $w$ with its iterative standard factorization a Lie $\dbZ$-polynomial $P_w$ in $X$.
Using a triangulation property for the Magnus homomorphism applied to $\tau_w$, combined with precise $p$-adic evaluations of binomial coefficients, we show that $(\sig_{w,n},\alp_{w',n})_n$ is the $\dbF_p$-coefficient of $w'$ in $P_w$.
Furthermore, $P_w$ is a linear combination of permutations of $w$, from which Theorem A follows.
Theorem B follows from the fact that $P_w$ is independent of $n$.
The proofs of both theorems are given in \S\ref{section on the transgression pairing}.

This alternative description of the entries $(\sig_{w,n},\alp_{w',n})_n$ enables us to compute them in an elementary way, yielding the table in \S\ref{section on the fundamental matrix of depth 4}.

\section{Lyndon words}
\label{section on Lyndon words}
Let $(X,\leq)$ be a totally ordered nonempty set, considered as an alphabet.
Let $X^*$ be the free unital monoid on $X$.
We consider its elements as (associative) words in the alphabet $X$ and write them as $w=(a_1\cdots a_s)$ with $a_1\nek a_s\in X$.
Thus the unit element $1$ of $X^*$ is the empty word and the product is the concatenation of words.
We write $|w|$ for the length $s$ of $w$.
The \textsl{support} of $w$ is the set
$\Supp(w)=\{a_1\nek a_s\}$.
When $w=uv$ for words $w,u,v$, we say that $u$ and $v$ are a \textsl{prefix}, resp., a \textsl{suffix}, of $w$.

Let $\leq_{\rm lex}$ be the lexicographic order on $X^*$ induced by $\leq$.
Let $\preceq$ be the \textsl{length-lexicographic order} on $X^*$, defined by $w\preceq v$ if and only if $|w|<|v|$ or else $|w|=|v|$ and $w\leq_{\rm lex} v$.

A nonempty word $w$ in $X^*$ is called a \textsl{Lyndon word} if it is smaller with respect to $\leq_{\rm lex}$ than all its proper suffixes.
The family of all Lyndon words in $X^*$ is a Hall family \cite{Reutenauer93}*{Th.\ 5.1}.
We write $\Lyn(X)$ (resp., $\Lyn_{\leq n}(X)$) for the set of all Lyndon words in $X^*$ (resp., of length $\leq n$).

\begin{exam}
\label{Lyndon words of length at most 4}
\rm
(1) \quad
All one-letter words $(x)$, $x\in X$, are Lyndon words.

(2) \quad
A two-letter word $(xy)$ in $X^*$ is Lyndon if and only if $x<y$.

(3) \quad
The three-letter Lyndon words are those of the forms 
$(xxy)$, $(xyy)$, $(xyz)$, $(xzy)$, where $x<y<z$.

(4) \quad 
The four-letter Lyndon words are those of the forms
\[
\begin{split}
&(xxxy), (xxyy), (xxyz), (xxzy), (xyxz), (xyyy), (xyyz), (xyzy), (xyzz),\\
&(xyzt), (xytz), (xzyy), (xzyz), (xzyt), (xzzy), (xzty), (xtyz), (xtzy)
\end{split}
\]
with $x<y<z<t$.
\end{exam}

Given a Lyndon word $w$ in $X^*$ with $|w|\geq2$, let $v$ be its longest proper suffix which is Lyndon, and write $w=uv$ with $u\in X^*$.
Since one-letter words are Lyndon, this is well-defined.
We call this decomposition the \textsl{standard factorization} of $w$. 
As noted in \cite{Efrat23}*{Lemma 2.2}, this definition coincides with the one given in \cite{Reutenauer93}*{\S4.1}, and therefore $u$ is also Lyndon.

\section{Power series}
\label{section on power series}
\subsection{Basic notions}
We fix a commutative unital ring $R$ and consider $X$ as a set of non-commuting variables.
Let $R\langle\langle X\rangle\rangle$ be the set of all formal power series in the variables of $X$ and with coefficients in $R$.
Thus $R\langle\langle X\rangle\rangle$ consists of all formal expressions $\sum_{w\in X^*}c_ww$ with $c_w\in R$ and where we view a word $w=(x_1\cdots x_s)$ also as the monomial $x_1\cdots x_s$. 
The set $R\langle\langle X\rangle\rangle$ forms an $R$-algebra in a standard way, where the multiplication is induced by the concatenation of words.
We define a Lie bracket on $R\langle\langle X\rangle\rangle$ as usual by $[f,g]=fg-gf$.

We write $R\langle\langle X\rangle\rangle^\times$ for the group of invertible elements in $R\langle\langle X\rangle\rangle$, and $R\langle\langle X\rangle\rangle^{\times,1}$ for the subset of $R\langle\langle X\rangle\rangle$ consisting of all power series $\sum_wc_ww$ with $c_{1}=1$.
Every such power series has an inverse $\sum_wd_ww$, where the coefficients $d_w$ are constructed by induction on $|w|$. 
Thus $R\langle\langle X\rangle\rangle^{\times,1}$ is a subgroup of $R\langle\langle X\rangle\rangle^\times$.

If the ring $R$ is profinite, i.e., an inverse limit of finite rings, then the addition and multiplication in $R\langle\langle X\rangle\rangle$ are continuous with respect to the product topology on $R^{X^*}$, making $R\langle\langle X\rangle\rangle$ a profinite ring.
Further, $R\langle\langle X\rangle\rangle^\times$ and $R\langle\langle X\rangle\rangle^{\times,1}$ are then profinite groups.

We write $R\langle X\rangle$ for the $R$-subalgebra of $R\langle\langle X\rangle\rangle$ consisting of all formal power series $\sum_{w\in X^*}c_ww$ for which the support $\{w\in X^*\ |\ c_w\neq0\}$ is finite.

The algebra $R\langle\langle X\rangle\rangle$ carries a canonical discrete valuation
\[
\nu=\nu_R\colon R\langle\langle X\rangle\rangle\to\dbZ_{\geq0}\cup\{\infty\},
\]
given by 
\[
\nu\Bigl(\sum_wc_ww\Bigr)=\min\bigl\{|w|\ \bigm|\ w\in X^*, c_w\neq0\bigr\}
\]
for $0\neq\sum_wc_ww\in R\langle\langle X\rangle\rangle$, and $\nu(0)=\infty$.
For $f,g\in R\langle\langle X\rangle\rangle$ one has
\[
\nu(fg)\geq\nu(f)+\nu(g), \quad \nu(f+g)\geq\min\{\nu(f),\nu(g)\}.
\]
The first inequality may be strict due to existence of zero divisors in $R$.

\subsection{Constructions using the standard factorization}
Assume further that the ordered set $(X,\leq)$ is finite and let $S=S_X$ be the free profinite group on the basis $X$ \cite{FriedJarden08}*{\S20.4}.

For every Lyndon word $w$, we define simultaneously a noncommutative polynomial $P_w\in\dbZ\langle X\rangle$ and an element $\tau_w\in S$ recursively using the standard factorization of $w$, as follows:

If $w=(x)$ is a one-letter word with $x\in X$ we set $P_{(x)}=x$ and $\tau_{(x)}=x$.
If $|w|\geq2$ we take the standard factorization $w=uv$ of $w$, where $u,v$ are Lyndon words of smaller length, and set
\[
P_w=[P_u,P_v]=P_uP_v-P_vP_u, \quad  \tau_w=[\tau_u,\tau_v]=\tau_u\inv\tau_v\inv\tau_u\tau_v.
\]
Thus $P_w$ is a Lie polynomial and $\tau_w$ is an iterated commutator.

\begin{prop}
\label{permutations}
Let $w$ be a Lyndon word in $X^*$.
\begin{enumerate}
\item[(a)]
The noncommutative polynomial $P_w$ is a $\dbZ$-linear combination of permutations of $w$.
\item[(b)]
The coefficient of $w$ in $P_w$ is $1$.
\end{enumerate}
\end{prop}
\begin{proof}
(a) \quad
We argue by induction on $s=|w|$.
If $s=1$, then $P_w=w$.

Assume that $s\geq2$ and let $w=uv$ be the standard factorization of $w$.
By the induction hypothesis, $P_u=\sum_{i=0}^ta_iu_i$ and $P_v=\sum_{j=0}^rb_jv_j$, where $u_i,v_j$ are permutations of $u$ and $v$, respectively, and $a_i,b_j\in\dbZ$.
Then
\[
P_w=[P_u,P_v]=\Bigl[\sum_{i=0}^ta_iu_i,\sum_{j=0}^rb_jv_j\Bigr]
=\sum_{i,j}a_ib_j[u_i,v_j]
=\sum_{i,j}a_ib_j(u_iv_j-v_ju_i).
\]
Here the concatenations $u_iv_j$ and $v_ju_i$ are permutations of $w=uv$.
Therefore $P_w$ is a $\dbZ$-linear combination of permutations of $w$.

\medskip

(b) \quad
This is contained in \cite{Reutenauer93}*{Th.\ 5.1}.
\end{proof}

For a ring $R$ we also consider $P_w$ as a polynomial in $R\langle X\rangle$ via the natural ring homomorphism $\dbZ\to R$.

\begin{cor}
\label{nu of Pw}
For every $w\in \Lyn(X)$ one has $\nu_R(P_w)=|w|$.
\end{cor}

Note that $\tau_w\in S^{(s,p)}$ where $s=|w|$.
When $1\leq s\leq n$ we therefore have $\tau_w^{p^{n-s}}\in S^{(n,p)}$.
Moreover, the following fact is proved in \cite{Efrat17}*{Th.\ 8.5}:

\begin{thm}
\label{profinite base}
The cosets of $\sig_{w,n}=\tau_w^{p^{n-s}}$, where $w$ ranges over all Lyndon words of length $1\leq s\leq n$ in $X^*$, form an $\dbF_p$-linear basis of $S^{(n,p)}/S^{(n+1,p)}$.
\end{thm}

\section{The Magnus homomorphism}
\label{section on the Magnus homomorphism} 
\subsection{Definition and basic properties}
Assume that $R$ is a profinite ring and let $S$ be again the free profinite group on the finite basis $X$.
Using the universal property of $S$, we define the (continuous) \textsl{Magnus homomorphism} $\Mag_R\colon S\to R\langle\langle X\rangle\rangle^{\times,1}$ by setting $\Mag_R(x)=1+x$ for $x\in X$.
For $\sig\in S$, let $\eps_{w,R}(\sig)$ be the coefficient of $w$ in $\Mag_R(\sig)$.
Thus
\[
\Mag_R(\sig)=\sum_{w\in X^*}\eps_{w,R}(\sig)w.
\]
As $\Mag_R$ preserves multiplication, for $w\in X^*$ and $\sig,\tau\in S$ one has
\begin{equation}
\label{eps of product}
\eps_{w,R}(\sig\tau)=\sum_{w=uv}\eps_{u,R}(\sig)\eps_{v,R}(\tau).
\end{equation}

In particular, if $|w|=1$, then $\eps_{w,R}\colon S\to R$ is a group homomorphism.

The following characterization of the lower $p$-central filtration $S^{(n,p)}$, $n=1,2\nek$ of $S$ is due to Koch \cite{Koch60} in the discrete case. 
See \cite{Efrat17}*{Lemma 4.1(b)} for the profinite case.

\begin{prop}
\label{lower p-central filtration via Magnus coefficients}
Let $n$ be a positive integer and let $\sig\in S$.
Then $\sig\in S^{(n,p)}$ if and only if $\eps_{w,\dbZ_p}(\sig)\in p^{n-|w|}\dbZ_p$ for every $w\in X^*$ with $1\leq |w|<n$.
\end{prop}

\subsection{Unipotent upper triangular matrices}
For $s\geq1$, let $\dbU_s(R)$ be the group of all unipotent upper-triangular $(s+1)\times(s+1)$-matrices with entries in $R$.
For a word $w=(a_1\cdots a_s)$ of length $s$ we consider the map 
\[
\rho_w\colon S\to\dbU_s(R), \quad\sig\mapsto\bigl[\eps_{(a_i\cdots a_{j-1}),R}(\sig)\bigr]_{1\leq i\leq j\leq s+1}.
\]
It follows from (\ref{eps of product}) that $\rho_w$ is a group homomorphism.
We call it the (profinite) \textsl{Magnus representation} associated with $w$ \cite{Efrat24}*{Prop.\ 4.1}.

Now assume that $1\leq s\leq n$ and take $R=\dbZ/p^{n-s+1}$.
By \cite{Efrat17}*{Prop.\ 6.3}, $\dbU_s(R)^{(n,p)}$ is the central subgroup of $\dbU_s(R)$ consisting of all matrices of the form $I_{s+1}+ap^{n-s}E_{1,s+1}$ with $a\in \dbZ$, i.e., the matrices which are $1$ on the main diagonal, $ap^{n-s}\pmod{p^{n-s+1}}$ at entry $(1,s+1)$, and zero elsewhere.
Thus $\dbU_s(R)^{(n,p)}\isom\dbZ/p$.
We set
\[
\overline{\dbU}_s(R)=\dbU_s(R)/\dbU_s(R)^{(n,p)}.
\]
By the Schreier correspondence  \cite{NeukirchSchmidtWingberg}*{Th.\ 1.2.4}, the central extension
\[
0\to\dbF_p\to\dbU_s(R)\to\overline{\dbU}_s(R)\to1
\]
corresponds to a cohomology element $\omega_{s,n}\in H^2(\overline{\dbU}_s(R),\dbF_p)$.

\subsection{The elements $\alp_{w,n}$}
The Magnus representation $\rho_w$ induces a profinite group homomorphism 
\[
\bar\rho_w\colon S/S^{(n,p)}\to\overline{\dbU}_s(R).
\]
We denote by 
\begin{equation}
\label{definition of alpha}
\alp_{w,n}=\bar\rho_w^*\omega_{s,n}
\end{equation}
the pullback of $\omega_{s,n}$ to $H^2(S/S^{(n,p)},\dbF_p)$ along $\bar\rho_w$.

\begin{exam}
\label{Bocksteins and Massey products}
\rm
(1) \quad
Suppose that $s=1$ and $w=(x)$.
We view $\eps_{(x),\dbZ/p^{n-1}}$ as an element of $H^1(S/S^{(n,p)},\dbZ/p^{n-1})=\Hom_{\rm cont}(S/S^{(n,p)},\dbZ/p^{n-1})$.
Let 
\[
\Bock\colon H^1(S/S^{(n,p)},\dbZ/p^{n-1})\to H^2(S/S^{(n,p)},\dbF_p)
\]
be the Bockstein map, i.e., the connecting homomorphism arising from the short exact sequence of trivial $S/S^{(n,p)}$-modules
\[
0\to \dbZ/p\to\dbZ/p^n\to\dbZ/p^{n-1}\to0. 
\]
Then $\alp_{(x),n}=\Bock(\eps_{(x),\dbZ/p^{n-1}})$ \cite{Efrat17}*{Example 7.4(a)}.

\medskip

(2) \quad
Suppose that $s=n$ and $w=(x_1\cdots x_n)$.
Then $\alp_{w,n}$ is an element of the $n$-fold Massey product $\langle\eps_{(x_1),\dbZ/p}\nek\eps_{(x_n),\dbZ/p}\rangle$ \cite{Efrat17}*{Example 7.4(b)}.
\end{exam}

In analogy with Theorem \ref{profinite base} we have:

\begin{thm}
[\cite{Efrat17}*{Th.\ 8.5}]
\label{cohomological base}
The elements $\alp_{w,n}$, where $w\in\Lyn_{\leq n}(X)$, form an $\dbF_p$-linear basis of $H^2(S/S^{(n,p)},\dbF_p)$.
\end{thm}

\section{Triangulation and binomial coefficients}
\label{section on triangualtion and binmomial coeffcients}
The following \textsl{triangulation property} will be fundamental for our analysis.
Here $R$ is a profinite commutative ring and $P_w$ is considered as before as a non-commutative polynomial in $R\langle X\rangle$.
We write $O(m)$ for a power series $f$ in $R\langle\langle X\rangle\rangle$ with $\nu_R(f)\geq m$.

\begin{prop} [\cite{Efrat17}*{Prop.\ 4.4(a)}]
\label{triangulation Efr17}
For $w\in\Lyn(X)$ one has
\[
\Mag_R(\tau_w)=1+P_w+O(|w|+1).
\]
\end{prop}

\begin{prop}
\label{Mag of powers of tau}
For $w\in\Lyn(X)$ and for any positive integer $k$ one has
\[
\Mag_R(\tau_w^{p^k})=1+p^kP_w+O(|w|+1).
\]
\end{prop}
\begin{proof}
By Proposition \ref{triangulation Efr17} and the binomial expansion formula,
\[
\begin{split}
\Mag_R(\tau_w^{p^k})&=(1+P_w+O(|w|+1))^{p^k}  \\
&=1+p^kP_w+O(|w|+1)+\sum_{i=2}^{p^k}\binom{p^k}i(P_w+O(|w|+1))^i.
\end{split}
\]
By Corollary \ref{nu of Pw}, $\nu_R(P_w)=|w|$.
Hence for every $2\leq i\leq p^k$ we have 
\[
\nu_R\Bigl(\binom{p^k}i(P_w+O(|w|+1))^i\Bigr)\geq i|w|\geq |w|+1
\]
and the assertion follows.
\end{proof}

For a nonnegative integer $k$, let $d_R(k)$ be the minimal integer $1\leq i\leq p^k$ such that $\binom{p^k}i\neq0$ in $R$.
We similarly have:

\begin{prop}
\label{Magnus and min}
Let $k\geq0$ be an integer and let $w\in\Lyn(X)$.
Then
\[
\Mag_R(\tau_w^{p^k})=1+O(d_R(k)|w|).
\]
\end{prop}
\begin{proof}
By Proposition \ref{triangulation Efr17},
\[
\Mag_R(\tau_w^{p^k})=(1+P_w+O(|w|+1))^{p^k}=1+\sum_{i\geq d_R(k)}\binom{p^k}i(P_w+O(|w|+1))^i.
\]
By Proposition \ref{permutations}, $\nu_R((P_w+O(|w|+1))^i)=i|w|$.
Therefore
\[
\begin{split}
\nu_R(\Mag_R(\tau_w^{p^k})-1)
&\geq\min\Bigl\{\nu_R((P_w+O(|w|+1))^i)\ \bigm|\ i\geq d_R(k)\Bigr\} \\
&=d_R(k)\cdot |w|.  \qedhere
\end{split}
\]
\end{proof}

We will need the following fact from \cite{Efrat20}*{Prop.\ 2.2(c)}:

\begin{prop}
\label{divisibility of binomial coefficients}
Let $t,k,m$ be positive integers such that $1\leq t\leq p^k$.
Then $p^m|\binom{p^k}i$ for every $1\leq i\leq t$ if and only if $\lfloor\log_pt\rfloor\leq k-m$.
\end{prop}

Given $1\leq m\leq k$ and taking $R=\dbZ/p^m$, this shows that $d_{\dbZ/p^m}(k)=p^{k-m+1}$.
We deduce from Proposition \ref{Magnus and min}:

\begin{cor}
\label{Magnus modulo p m}
Let $1\leq m\leq k$ and let $w\in\Lyn(X)$.
Then
\[
\Mag_{\dbZ/p^m}(\tau_w^{p^k})=1+O(p^{k-m+1}|w|).
\]
\end{cor}

\section{The pairing $\langle w,w'\rangle_n$}
\label{section on the pairing}
\subsection{Definition and vanishing criteria} 
We fix a positive integer $n$.

For positive integers $k,m$ we denote the ideal $p^k(\dbZ/p^m)$ in the ring $\dbZ/p^m$ by $p^k\dbZ/p^m$.
Given $1\leq s\leq n$ there is a well-defined group isomorphism
\[
\iota_{n,s}\colon p^{n-s}\dbZ/p^{n-s+1}\xrightarrow{\sim}\dbZ/p, \quad
p^{n-s}a\!\!\!\!\pmod{p^{n-s+1}}\mapsto a\!\!\!\!\pmod p,
\]
for $a\in\dbZ$.

Now let $w,w'\in X^*$ be words of lengths $1\leq s,s'\leq n$, respectively, with $w$ Lyndon.
Recall that $\sig_{w,n}=\tau_w^{p^{n-s}}\in S^{(n,p)}$.
By Proposition \ref{lower p-central filtration via Magnus coefficients}, $\eps_{w',\dbZ/p^{n-s'+1}}(\sig_{w,n})\in p^{n-s'}\dbZ/p^{n-s'+1}$.
We may therefore define
\begin{equation}
\label{definition of pairing}
\langle w,w'\rangle_n=\iota_{n,s'}(\eps_{w',\dbZ/p^{n-s'+1}}(\sig_{w,n}))\in\dbZ/p.
\end{equation}

The following properties of the pairing $\langle w,w'\rangle_n$ were proved in \cite{Efrat17}*{Prop.\ 6.4}.
Recall that $\preceq$ denotes the length-lexicographic order.

\begin{prop}
\label{basic properties of the pairing}
Let $w,w'$ be words in $X^*$ of lengths $1\leq s,s'\leq n$, respectively, with $w$ Lyndon.
\begin{enumerate}
\item[(a)]
If $w'\prec w$, then  $\langle w,w'\rangle_n=0$;
\item[(b)]
$\langle w,w\rangle_n=1$;
\item[(c)]
If $\Supp(w')\not\subseteq\Supp(w)$, then  $\langle w,w'\rangle_n=0$.
\end{enumerate}
\end{prop}

We wish to show that $\langle w,w'\rangle_n$ vanishes in more situations.

\begin{prop}
\label{zero pairing for distinct lengths}
Let $w\in\Lyn(X)$ and $w'\in X^*$ satisfy $|w|<|w'|\leq n$.
Then $\langle w,w'\rangle_n=0$, with the unique exception $p=2$, $w=(x)$ and $w'=(xx)$ for some $x\in X$, in which case $\langle w,w'\rangle_n=-1$.
\end{prop}
\begin{proof}
By Corollary \ref{Magnus modulo p m} with $k=n-|w|$ and $m=n-|w'|+1$,
\begin{equation}
\label{Magnus application}
\Mag_{\dbZ/p^{n-|w'|+1}}(\tau_w^{p^{n-|w|}})=1+O(p^{|w'|-|w|}|w|).
\end{equation}

A direct computation shows that the function $f_p(x)=x/p^x$ is decreasing in $[1/\ln p,\infty)$.
In particular, it is decreasing in $[1,\infty)$ when $p>2$, and in $[2,\infty)$ when $p=2$. 
Moreover, when $p=2$ we have $f_2(1)=f_2(2)$. 
It follows that, if $(p,|w|,|w'|)\neq(2,1,2)$, then $|w'|<p^{|w'|-|w|}|w|$.
Additionally, the assumptions imply that $w'$ is not the empty word.
We conclude from (\ref{Magnus application}) that $\eps_{w',\dbZ/p^{n-|w'|+1}}(\sig_{w,n})=0$.
By (\ref{definition of pairing}), $\langle w,w'\rangle_n=0$ in this case.

Further, if $\Supp(w')\not\subseteq\Supp(w)$, then $\langle w,w'\rangle_n=0$, by Proposition \ref{basic properties of the pairing}(c).

It remains to consider the case where $(p,|w|,|w'|)=(2,1,2)$ and $\Supp(w')\subseteq\Supp(w)$.
Then $w=(x)$ and $w'=(xx)$ for some $x\in X$, which is the exceptional case in the assertion.
In this case we compute:
\[
\Mag_{\dbZ/p^{n-|w'|+1}}(\tau_w^{p^{n-|w|}})
=\Mag_{\dbZ/2^{n-1}}((x)^{2^{n-1}})=(1+x)^{2^{n-1}}
=\sum_{i=0}^{2^{n-1}}\binom{2^{n-1}}ix^i.
\]
Since $n\geq|w'|=2$, the $\dbZ/2^{n-1}$-coefficient of $w'=(xx)$ here is
\[
\binom{2^{n-1}}2=2^{n-2}(2^{n-1}-1)\equiv-2^{n-2}\pmod{2^{n-1}}.
\]
By (\ref{definition of pairing}), this shows that $\langle w,w'\rangle_n=-1$.
\end{proof}

\begin{cor}
\label{cor on vanishing of pairing}
If $w,w'\in\Lyn(X)$ and $|w|<|w'|\leq n$, then $\langle w,w'\rangle_n=0$.
\end{cor}
\begin{proof}
Apply Proposition \ref{zero pairing for distinct lengths}, noting that $(xx)$, $x\in X$, is not Lyndon.
\end{proof}

\subsection{$\langle w,w'\rangle_n$ as a coefficient of $P_w$}
\begin{thm}
\label{pairing as a coefficient}
Let $w,w'\in\Lyn_{\leq n}(X)$. 
Then $\langle w,w'\rangle_n$ is the coefficient $(P_w)_{w'}$ of $w'$ in $P_w$ modulo $p$.
\end{thm}
\begin{proof}
First assume that $|w|\neq|w'|$.
Then $(P_w)_{w'}=0$, by Proposition \ref{permutations}(a).
If $|w'|<|w|$, then by Proposition \ref{basic properties of the pairing}(a), $\langle w,w'\rangle_n=0$.
If $|w|<|w'|$, then the same holds by Corollary \ref{cor on vanishing of pairing}.
This proves the theorem in this case.

Now consider the case where $s:=|w|=|w'|$.
Then, by Proposition \ref{Mag of powers of tau}, 
\[
\Mag_{\dbZ/p^{n-s+1}}(\tau_w^{p^{n-s}})=1+p^{n-s}P_w+O(s+1).
\]
It follows that
\[
\eps_{w',\dbZ/p^{n-s+1}}(\tau_w^{p^{n-s}})=p^{n-s}(P_w)_{w'}\!\!\!\pmod{p^{n-s+1}}\in p^{n-s}\dbZ/p^{n-s+1}.
\]
Hence, by (\ref{definition of pairing}),
\[
\langle w,w'\rangle_n=\iota_{n,s}(\eps_{w',\dbZ/p^{n-s+1}}(\tau_w^{p^{n-s}}))
=(P_w)_{w'}\pmod p.
\qedhere
\]
\end{proof}

\begin{cor}
\label{vanishing corollary}
If $w,w'\in \Lyn_{\leq n}(X)$ are not permutations of each other, then $\langle w,w'\rangle_n=0$.
\end{cor}
\begin{proof}
This follows from Theorem \ref{pairing as a coefficient} and from Proposition \ref{permutations}(a).
\end{proof}

\begin{rem}
\rm
For Lyndon words $w$ and $w'$ which are permutations of one another, it is possible for $\langle w,w'\rangle_n$ to be nonzero as well as zero.
For instance, in \S\ref{section on the fundamental matrix of depth 4} we show that for $x,y,z\in X$ with $x<y<z$ one has
\[
\langle (xyz),(xzy)\rangle_3=-1, \quad \langle (xxyz),(xyxz)\rangle_4=0.
\]
\end{rem}

As an additional consequence of Theorem \ref{pairing as a coefficient} we obtain:

\begin{cor}
\label{indpendence of n}
For $w,w'\in\Lyn(X)$, the value of $\langle w,w'\rangle_n$ is independent of $n$, provided $n\geq|w|,|w'|$.
\end{cor}
Thus we may simply write $\langle w,w'\rangle$ for this pairing.

\section{The transgression pairing}
\label{section on the transgression pairing}
\subsection{Cohomological constructions}
Let $G$ be a profinite group and $K$ a closed normal subgroup of $G$.
We identify $H^1(K,\dbF_p)$ with $\Hom_{\rm cont}(K,\dbF_p)$ and let $G$ act on $H^1(K,\dbF_p)$ as usual by $({}^g\psi)(k)=\psi(g\inv kg)$ for $g\in G$, $k\in K$ and $\psi\in H^1(K,\dbF_p)$.
There is a well-defined substitution pairing
\[
K/K^p[G,K]\times H^1(K,\dbF_p)^G\to\dbF_p, \quad (\bar k,\psi)\mapsto \psi(k)
\]
which is non-degenerate, i.e., its left and right kernels are trivial \cite{EfratMinac11}*{Cor.\ 2.2}.
In particular, take $K=G^{(n,p)}$ to be the $n$th term of the lower $p$-central filtration of $G$.
We obtain a non-degenerate bilinear map 
\[
G^{(n,p)}/G^{(n+1,p)}\times H^1(G^{(n,p)},\dbF_p)^G\to\dbF_p.
\]

Now let $S$ be as before the free profinite group on the finite basis $X$ and consider the exact five term sequence of cohomology groups
\[
\begin{split}
0\to H^1(S/S^{(n,p)},\dbF_p)&\xrightarrow{\inf}H^1(S,\dbF_p)\xrightarrow{\res} H^1(S^{(n,p)},\dbF_p)^S \\
&\xrightarrow{\trg} H^2(S/S^{(n,p)},\dbF_p)\xrightarrow{\inf} H^2(S,\dbF_p)
\end{split}
\]
\cite{NeukirchSchmidtWingberg}*{Prop.\ 1.6.7}.
When $n\geq2$ we have $S^{(n,p)}\leq S^{(2,p)}=S^p[S,S]$, so the first map is an isomorphism.
Further, $H^2(S,\dbF_p)=0$ \cite{NeukirchSchmidtWingberg}*{Th.\ 3.5.6}.
Therefore the transgression map $\trg$ above is an isomorphism.
We obtain the bilinear \textsl{transgression pairing}
\[
(\cdot,\cdot)_n\colon S^{(n,p)}\times H^2(S/S^{(n,p)},\dbF_p)\to\dbF_p, \quad
(\sig,\alp)_n=-(\trg\inv(\alp))(\sig)
\]
for $\sig\in S^{(n,p)}$.
Its left kernel is $S^{(n+1,p)}$ and its right kernel is trivial.

\subsection{The fundamental matrix}
The pairing $(\cdot,\cdot)_n$ is connected to the pairing $\langle\cdot,\cdot\rangle_n$ of \S\ref{section on the pairing} as follows:

\begin{thm}
[\cite{Efrat17}*{Cor.\ 8.2}]
\label{connection between the pairings}
Let $w,w'$ be words in $X^*$ of lengths $1\leq |w|,|w'|\leq n$, respectively, with $w$ Lyndon.
Then
\[
(\sig_{w,n},\alp_{w',n})_n=\langle w,w'\rangle_n.
\]
\end{thm}

\medskip

\begin{proof}[Proof of Theorem A]
This follows from Theorem \ref{connection between the pairings} and Corollary \ref{vanishing corollary}.
\end{proof}

In view of Theorems \ref{profinite base} and \ref{cohomological base}, the \textsl{fundamental matrix} $[(\sig_{w,n},\alp_{w',n})_n]_{w,w'}$ \textsl{of depth $n$}, where $w,w'\in \Lyn_{\leq n}(X)$, encodes the full information about the transgression pairing $(\cdot,\cdot)_n$.
To make it well-defined one needs to choose an order on the index set $\Lyn_{\leq n}(X)$. \cite{Efrat17} uses the length-lexicographic order, making the matrix unipotent upper triangular (by Proposition \ref{basic properties of the pairing} and Theorem \ref{connection between the pairings}).
However, in light of Theorem A, it is more natural to order the words in $\Lyn_{\leq n}(X)$ with respect to the following refined lexicographic order:
First order the permutation classes according to the length-lexicographic order of their lexicographically least elements. 
Within each permutation class, use the lexicographic order.

\begin{cor}
Under this order on $\Lyn_{\leq n}(X)$, the fundamental matrix $[(\sig_{w,n},\alp_{w',n})_n]_{w,w'}$ is block diagonal, with one diagonal block for each permutation class, and each diagonal block is unipotent upper triangular. 
\end{cor}

\begin{proof}[Proof of Theorem B]
This follows from Theorem \ref{connection between the pairings} and Corollary \ref{indpendence of n}.
\end{proof}

\begin{rem}
\rm
If $w=(x_1\cdots x_s)$ is Lyndon, then $x_1=\min(\Supp(w))$.
Therefore there are at most $(s-1)!$ permutations of $w$ which are Lyndon.

Further, let $x_1<\cdots<x_s$ in $X$ and let $v$ be any permutation of $(x_2\cdots x_s)$.
Then the concatenation $(x_1)v$ is a permutation of $w=(x_1x_2\cdots x_s)$ and is Lyndon.
Therefore the permutation class of $w$ consists of $(s-1)!$ Lyndon words.
Thus, when $|X|\geq n$, the fundamental matrix has a diagonal block of maximal size $(n-1)!\times(n-1)!$.  
\end{rem}

\section{The fundamental matrix of depth $4$}
\label{section on the fundamental matrix of depth 4}
The Lyndon words of length $\leq4$ were listed in Example \ref{Lyndon words of length at most 4}.
We compute the corresponding blocks of the fundamental matrix.
Here $x<y<z<t$ and we do not distinguish between words of the same pattern, e.g., $(xxyz)$ and $(yyzt)$.
Permutation classes of size $1$ give the block $[1]$.
The permutation classes containing more than one Lyndon word are the following classes $C_1,C_2,C_3,C_4,C_5$.
We order them lexicographically, writing each word in its iterative standard factorization.
We write down the polynomials $P_w$ needed for the computation and the  matrix block $B_i$ of $C_i$.

\goodbreak
\bigskip

\hrule

\medskip

$C_1=\{(x(yz)), ((xz)y)\}$

$P_{(xyz)}=xyz-xzy-yzx+zyx$

\[
B_1=\begin{bmatrix}
1&-1\\
0&\phm1\\
\end{bmatrix}
\]

\medskip

\hrule

\medskip

$C_2=\{(x(x(yz))),(x((xz)y)),((xy)(xz))\}$  

$P_{(xxyz)}=xxyz-xxzy-2xyzx+2xzyx+yzxx-zyxx$

$P_{(xxzy)}=xxzy-xyxz+xyzx-xzxy-xzyx+yxzx-yzxx+zxyx$

\[
B_2=\begin{bmatrix}
1&-1&\phm0\\
0&\phm1&-1\\
0&\phm0&\phm1
\end{bmatrix}
\]

\medskip

\hrule

\medskip

$C_3=\{(x(y(yz))),((x(yz))y),(((xz)y)y)\}$

$P_{(xyyz)}=xyyz-2xyzy+xzyy-yyzx+2yzyx-zyyx$  

$P_{(xyzy)}=xyzy-xzyy-yxyz+yxzy+yyzx-yzxy-yzyx+zyxy$ 

\[
B_3=\begin{bmatrix}
1&-2&\phm1\\
0&\phm1&-1\\
0&\phm0&\phm1
\end{bmatrix}
\]

\medskip

\hrule

\medskip

$C_4=\{(x((yz)z)), ((xz)(yz)), (((xz)z)y)\}$

$P_{(xyzz)}=xyzz-2xzyz+xzzy-yzzx+2zyzx-zzyx$ 

$P_{(xzyz)}=xzyz-xzzy-yzxz+yzzx-zxyz+zxzy+zyxz-zyzx$

\[
B_4=\begin{bmatrix}
1&-2&\phm1\\
0&\phm1&-1\\
0&\phm0&\phm1
\end{bmatrix}
\]
\medskip

\hrule

\medskip

$C_5=\{(x(y(zt))),(x((yt)z)),((xz)(yt)),((x(zt))y),((xt)(yz)),(((xt)z)y)\}$

$P_{(xyzt)}=xyzt-xytz-xzty+xtzy-yztx+ytzx+ztyx-tzyx$

$P_{(xytz)}=xytz-xzyt+xzty-xtyz-ytzx+zytx-ztyx+tyzx$ 

$P_{(xzyt)}=xzyt-xzty-ytxz+ytzx-zxyt+zxty+tyxz-tyzx$

$P_{(xzty)}=xzty-xtzy-yxzt+yxtz+yztx-ytzx-ztxy+tzxy$

$P_{(xtyz)}=xtyz-xtzy-yzxt+yztx+zyxt-zytx-txyz+txzy$

\[
B_5=\begin{bmatrix}
1&-1&\phm0&-1&\phm0&\phm1\\
0&\phm1&-1&\phm1&-1&\phm0\\
0&\phm0&\phm1&-1&\phm0&\phm0\\
0&\phm0&\phm0&\phm1&\phm0&-1\\
0&\phm0&\phm0&\phm0&\phm1&-1\\
0&\phm0&\phm0&\phm0&\phm0&\phm1
\end{bmatrix}
\]

\medskip

\hrule


\begin{bibdiv}
\begin{biblist}



\bib{Efrat17}{article}{
   author={Efrat, Ido},
   title={The cohomology of canonical quotients of free groups and Lyndon words},
   journal={Doc. Math.},
   volume={22},
   date={2017},
   pages={973--997},
%   issn={1431-0635},
%   review={\MR{3665398}},
}

\bib{Efrat20}{article}{
   author={Efrat, Ido},
   title={The lower $p$-central series of a free profinite group and the
   shuffle algebra},
   journal={J. Pure Appl. Algebra},
   volume={224},
   date={2020},
%   number={6},
   pages={106260, 13},
%   issn={0022-4049},
%   review={\MR{4048520}},
%   doi={10.1016/j.jpaa.2019.106260},
}

\bib{Efrat23}{article}{
   author={Efrat, Ido},
   title={The $p$-Zassenhaus filtration of a free profinite group and shuffle relations},
   journal={J. Inst. Math. Jussieu},
   volume={22},
   date={2023},
   number={2},
   pages={961--983},
%   issn={1474-7480},
%   review={\MR{4557910}},
%   doi={10.1017/S1474748021000426},
}



\bib{Efrat24}{article}{
   author={Efrat, Ido},
   title={Cohomology and the combinatorics of words for Magnus formations},
   journal={New York J. Math.},
   volume={30},
   date={2024},
   pages={1177--1195},
%   review={\MR{4791059}},
}

	

\bib{Efrat26}{article}{
author={Efrat, Ido},
title={Mild pro-p groups and ordered monoids},
date={2026},
eprint={https://arxiv.org/abs/2604.25789},
}


\bib{EfratMinac11}{article}{
   author={Efrat, Ido},
   author={Min\'a\v c, J\'an},
   title={On the descending central sequence of absolute Galois groups},
   journal={Amer. J. Math.},
   volume={133},
   date={2011},
   number={6},
   pages={1503--1532},
}




\bib{FriedJarden08}{book}{
   author={Fried, Michael D.},
   author={Jarden, Moshe},
   title={Field Arithmetic},
   edition={4},
   publisher={Springer, Cham},
   date={2023},
   pages={xxiv+792},
}



\bib{HamzaMaireMinacTan25}{article}{
   author={Hamza, Oussama},
   author={Min\'a\v c, Jan},
   author={Maire, Christian},
   author={T\^an, N.D.},
   title={Maximal 2-extensions of Pythagorean fields and Right Angled Artin Groups},
   date={2025},
   eprint={arXiv:2510.11970},
   }


\bib{Koch60}{article}{
    author={Koch, H.},
     title={\"Uber die Faktorgruppen einer absteigenden Zentralreihe},
     journal={Math.\ Nachr.},
     volume={22},
     date={1960},
     pages={159\ndash161},
}


\bib{Labute67}{article}{
author={Labute, John},
title={Classification of Demushkin groups},
journal={Canad.\ J.\ Math.},
volume={19},
date={1967},
pages={106\ndash132},
}


   
\bib{MinacPasiniQuadrelliTan21}{article}{
   author={Min\'a\v c, Jan},
   author={Pasini, Federico William},
   author={Quadrelli, Claudio},
   author={T\^an, N.D.},
   title={Koszul algebras and quadratic duals in Galois cohomology},
   journal={Adv. Math.},
   volume={380},
   date={2021},
   pages={Paper No. 107569, 49},
%   issn={0001-8708},
%   review={\MR{4200471}},
%   doi={10.1016/j.aim.2021.107569},
}

\bib{MinacPasiniQuadrelliTan22}{article}{
   author={Min\'a\v c, J.},
   author={Pasini, F. W.},
   author={Quadrelli, C.},
   author={T\^an, N. D.},
   title={Mild pro-$p$ groups and the Koszulity conjectures},
   journal={Expo. Math.},
   volume={40},
   date={2022},
   number={3},
   pages={432--455},
%   issn={0723-0869},
%   review={\MR{4475389}},
%   doi={10.1016/j.exmath.2022.03.004},
}


\bib{MinacTan17}{article}{
   author={Min\'a\v c, J\'an},
   author={T\^an, N.D.},
   title={Triple Massey products and Galois theory},
   journal={J. Eur. Math. Soc. (JEMS)},
   volume={19},
   date={2017},
   number={1},
   pages={255--284},
%   issn={1435-9855},
%   review={\MR{3584563}},
%   doi={10.4171/JEMS/665},
}




\bib{NeukirchSchmidtWingberg}{book}{
  author={Neukirch, J{\"u}rgen},
  author={Schmidt, Alexander},
  author={Wingberg, Kay},
  title={Cohomology of Number Fields, Second edition},
  publisher={Springer},
  place={Berlin},
  date={2008},
}



\bib{Quadrelli24}{article}{
   author={Quadrelli, Claudio},
   title={Massey products in Galois cohomology and the elementary type    conjecture},
   journal={J. Number Theory},
   volume={258},
   date={2024},
   pages={40--65},
%   issn={0022-314X},
%   review={\MR{4685899}},
%   doi={10.1016/j.jnt.2023.11.002},
}



\bib{Reutenauer93}{book}{
   author={Reutenauer, Christophe},
   title={Free Lie Algebras},
   series={London Mathematical Society Monographs. New Series},
   volume={7},
   note={Oxford Science Publications},
   publisher={The Clarendon Press, Oxford University Press, New York},
   date={1993},
   pages={xviii+269},
}


\bib{SerreDemuskin}{article}{
   author={Serre, Jean-Pierre},
   title={Structure de certains pro-$p$-groupes (d'apr\`es Demu\v skin)},
   conference={
      title={S\'eminaire Bourbaki (1962/63), Exp.\ 252},
   },
   label={Ser63},
}





\bib{VogelThesis}{thesis}{
author={Vogel, Denis},
title={Massey products in the Galois cohomology of number fields},
type={Ph.D.\ thesis},
place={Universit\"at Heidelberg},
date={2004},
}



\end{biblist}
\end{bibdiv}
\end{document}